\documentclass[12pt]{amsart}
\usepackage{amsmath,amssymb,amsbsy,amsfonts,amsthm,latexsym,amsopn,amstext,amsxtra,epic, euscript,amscd,indentfirst}
\usepackage{graphicx}
\usepackage{enumerate}
\usepackage{tikz, tikz-3dplot}
\begin{document}
\renewcommand{\emptyset}{\varnothing}
\newtheorem{theorem}{Theorem}
\newtheorem{conjecture}[theorem]{Conjecture}
\newtheorem{proposition}[theorem]{Proposition}
\newtheorem{question}[theorem]{Question}
\newtheorem{lemma}[theorem]{Lemma}
\newtheorem{cor}[theorem]{Corollary}
\newtheorem{obs}[theorem]{Observation}
\newtheorem{proc}[theorem]{Procedure}
\newtheorem{defn}[theorem]{Definition}
\newcommand{\comments}[1]{} 
%% DEFINITIONS
\def\Z{\mathbb Z}
\def\Za{\mathbb Z^\ast}
\def\Fq{{\mathbb F}_q}
\def\R{\mathbb R}
\def\N{\mathbb N}
\def\cH{\overline{\mathcal H}}
\def\cF{\mathcal F}
\def\C{\mathbb C}
\def\P{\mathcal P}

\title[Counting Clean Triangles]{To count clean triangles\\ we count on $imph(n)$}
\author{Mizan R. Khan}
\address{MRK: Department of Mathematical Sciences, Eastern Connecticut State University, Willimantic, CT 06226}
\email{khanm@easternct.edu}

\author{Riaz R. Khan}
\address{RRK: Johannesburg, South Africa}
\email{riaz.r.khan@gmail.com}

\date{\today}

\maketitle

\begin{center} \emph{Dedicated to Ammi on her 89th birthday} \end{center}

\begin{abstract} 
A clean lattice triangle in $\R^2$ is a triangle that does not contain any lattice points on its sides other than its vertices. The central goal of this paper is to count the number of clean triangles of a given area up to unimodular equivalence. In doing so we use a variant of the Euler phi function which we call $imph(n)$ (imitation phi). Nothing in this article is original. In addition to Scott's~\cite{Sco} work, almost certainly all of this material is squirrelled away in some of Bruce Reznick's papers. However our hope is that after skimming this article our dynamic, energetic and enthusiastic reader (DEER in short) will try to discover a suitable arithmetic function called $oomph$ and a suitable class of arithmetic functions called $oops$.
\end{abstract}

\section{Introduction}

This expository article has two goals. The first (and minor) goal is to publicize a beautiful inequality for convex lattice polygons discovered by Scott~\cite{Sco} in the mid seventies. In doing so we arrive at a situation where, with a little more work, we can determine the number of clean lattice triangles of fixed area up to unimodular equivalence. This count is the central theme of our article. Throughout the article we will only work with lattice triangles in $\R^2$ with the lattice being $\Z^2$. The remainder of the introduction is devoted to listing some definitions and stating two well known theorems --- Pick's theorem and Burnside's lemma.

\begin{defn}
A lattice triangle is said to be \emph{clean} if the only lattice points on its sides are the vertices. If a clean lattice triangle does not contain any lattice in its interior, then we call it an \emph{empty} triangle.
\end{defn}

\noindent {\bf Remarks}: A lattice triangle in $\R^2$ is empty if and only if it has area 1/2. This does not hold in higher dimensions. Indeed there are no bounds on the volume of empty lattice tetrahedra. The terms \emph{clean} and \emph{empty} are due to Reznick~\cite{Rez}. In the literature empty lattice tetrahedra are sometimes called fundamental or primitive or Reeve.

\begin{defn} 
An \emph{affine unimodular map} is an affine map 
$$L: \R^n \rightarrow \R^n \textrm{ of the form } L(\vec{x}) = M\vec{x} + \vec{u},$$
where $M \in GL_n(\Z)$, $\det(M) =\pm 1$, and $\vec{u}\in \Z^n$. 
\end{defn}

The dynamic, energetic and enthusiastic reader (DEER) should note that affine unimodular maps preserve the lattice. We caution DEER that on occasion we omit the word \emph{affine} and simply say \emph{unimodular map}. In such an event DEER should keep in mind that we may be translating by a non-zero vector. If we want to emphasize that we are only acting by an unimodular matrix and there is no translation, then we will use the term \emph{unimodular transformation}.

\begin{defn}
Let $P_1, P_2$ be two lattice polygons. We say that $P_1$ and $P_2$ are \emph{unimodularly equivalent} if there is an affine unimodular map $L$ such that $L(P_1)=P_2$. 
\end{defn}

We will need to invoke Pick's theorem in a couple of places. This is a lovely elementary result that relates the area of a lattice polygon to the number of lattice points in the polygon (including the sides). 

\begin{theorem}[Pick]
Let $C$ be a simple polygon in $\R^2$ with vertices in $\Z^2$. Let 
$$I = \#(\Z^2 \cap interior(C)) \textrm{ and } B= \#(\Z^2 \cap boundary(C)).$$
Then,
\begin{equation}
area(C) = I+ B/2-1.
\end{equation}
\end{theorem}

From Pick's theorem we immediately obtain two facts. The first being that any lattice triangle in $\R^2$ is empty if and only if it has area 1/2. The second is that the area of a clean triangle has to be of the form $n/2$ with $n$ odd. Finally we state the combinatorial result which is usually called Burnside's lemma. It will deliver the coup de grace for our counting formula.

\begin{theorem}[Burnside's lemma]
Let $G$ be a finite group that acts on a finite set $X$. Then the number of distinct orbits, $N$, for this action is
$$ N= \frac{1}{\#G}\sum_{g\in G} \#Fix(g),$$
where $Fix(g)= \{x \in X \, : \, g(x)=x \}$, that is, $Fix(g)$ is the set of fixed points of $g$.
\end{theorem}

\section{Scott's Inequality for Lattice Triangles}

We begin by showing that any lattice triangle in $\R^2$ is unimodularly equivalent to a lattice triangle with a horizontal base. Scott~\cite{Sco} combined this observation with Pick's theorem to discover an inequality between the number of interior lattice points and boundary lattice points for convex lattice polygons that contain at least one interior lattice point. In the interest of keeping the proofs short we restrict our proof of Scott's theorem to lattice triangles. We  then return to this lemma to implement our approach to counting clean triangles of fixed area. Thus it is reasonable to say that Lemma~\ref{unimod-map} forms the foundation of both of our main results.

\begin{lemma}\label{unimod-map}
Let $\Delta$ be a lattice triangle in $\R^2$. Then $\Delta$ is unimodularly equivalent  to a lattice triangle $\triangle$ with vertices $(0,0), (b,0)$ and $(m,h)$ with the integers $b,m,h$ satisfying the following inequalities. 

\begin{itemize}
\item $ b,h > 0; m\ge 0; m< h;$
\item $b \ge \gcd(m,h), \gcd(m-b,h).$
\end{itemize}

\noindent $($Remark: The inequality  $b \ge \gcd(m,h), \gcd(m-b,h)$ ensures that the horizontal base of $\triangle$ does {\bf not} contain fewer lattice points than either of the other two sides of $\triangle$.$)$
\end{lemma}

\begin{proof} 
Since unimodular equivalence permits translation by lattice points, we may assume, without loss of generality, that the vertices of $\Delta$ are $(0,0),\vec{u}= (u_1,u_2)$ and $\vec{v}=(v_1,v_2).$ Furthermore, we can also assume that the side consisting of the vertices $(0,0)$ and $\vec{u}$ does {\bf not} contain fewer lattice points than either of the other two sides of $\Delta$. (If this is not the case, then we simply adjust by the appropriate translation.)  

We now consider the lattice point $\vec{w} = (u_1,u_2)/\gcd(u_1,u_2)=(w_1,w_2)$. Since $\gcd(w_1,w_2)=1$, by the extended Euclidean Algorithm there exist $x_1,x_2\in \Z$ such that 
$$x_2w_1-x_1w_2=1.$$
We now consider the unimodular map 
$T_0 : \R^2 \rightarrow \R^2$ via 
$$T_0(\vec{w}) = (1,0), T_0(\vec{x}) =(0,1).$$

If $T_0(\vec{v})$ lies in the lower half plane, then we simply apply the unimodular transformation $T_1 : \R^2 \rightarrow \R^2$ given by the matrix 
\begin{equation*}
 \begin{pmatrix}1 & 0 \\ 0 & -1 \\ \end{pmatrix}
\end{equation*}
and obtain that $T_1 \circ T_0 (\vec{u}) =\gcd(u_1,u_2)\cdot e_1$ and $T_1 \circ T_0(\vec{v})$ lies in the upper half plane. 

We have shown the existence of an unimodular map that sends $\vec{w} \mapsto (1,0)$ and $\vec{v} \mapsto \vec{a}= (a_1,a_2)$ with $a_2 \in \Z^+$. We set 
\begin{equation}
b =\gcd(u_1,u_2) \textrm{ and } h = a_2.
\end{equation}

If $0\le a_1 <a_2=h$, then we simply set $m =a_1$. What if we have that $a_1 <0$ or $a_1 \ge h$? In either case we apply the division algorithm to express 
$$ a_1 = qh +r \textrm{ with } 0\le r < h.$$
We then apply the unimodular transformation $T_2 : \R^2 \rightarrow \R^2$ given by the matrix 
\begin{equation*}
 \begin{pmatrix}1 & -q \\ 0 & 1 \\ \end{pmatrix}.
\end{equation*}
and obtain that $T_2((b,0))=(b,0)$ and $T_2((a_1,h))=(r,h)$. We then set $m=r$. Thus in all cases some combination of $T_0,T_1,T_2$ will do the trick.
\end{proof}

An example of this result is that the triangle with vertices $(0,0), (-3,-3)$ and $(2, 4)$ is unimodularly equivalent to the triangle with vertices $(0,0),(3,0)$ and $(2,2)$. See Figure~\ref{fig:scott triangle}. 

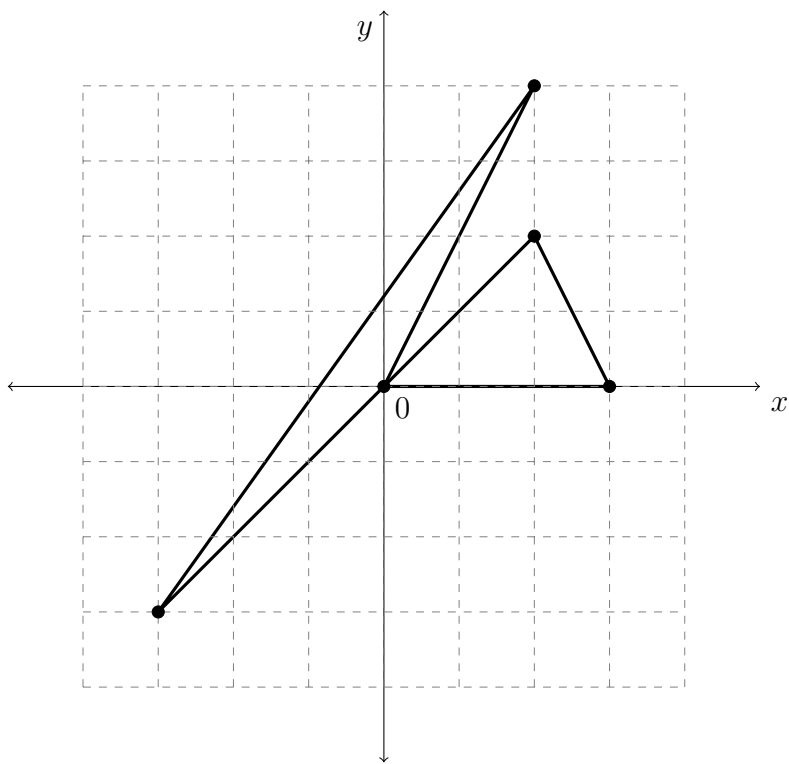
\begin{figure}
\begin{center}
\begin{tikzpicture}[scale=1]
\coordinate [label=below right:0] (0) at (0,0);
\coordinate [label=below right:$x$] ($x$) at (5,0);
\coordinate [label=below left:$y$] ($y$) at (0,5);

\draw[<->] (-5,0) -- (5,0);
\draw[<->] (0,-5) -- (0,5) ;
\draw [very thick] (0,0) --(3,0) -- (2,2) -- (0,0);
\draw [very thick] (0,0) --(-3,-3) -- (2,4) -- (0,0);
\draw[dashed,help lines] (-4,-4) grid (4,4);
\draw[fill] (2,2) circle [radius=0.08];
\draw[fill] (0,0) circle [radius=0.08];
%\draw[fill] (1,0) circle [radius=0.08];
%\draw[fill] (2,0) circle [radius=0.08];
\draw[fill] (3,0) circle [radius=0.08];
%\draw[fill] (1,1) circle [radius=0.08];
%\draw[fill] (2,1) circle [radius=0.08];
%\draw[fill] (0,1) circle [radius=0.08];
%\draw[fill] (3,1) circle [radius=0.08];
%\draw[fill] (3,2) circle [radius=0.08];
\draw[fill] (2,4) circle [radius=0.08];
\draw[fill] (-3,-3) circle [radius=0.08];

\end{tikzpicture}
\end{center}
\caption{Two triangles that are unimodularly equivalent}
\label{fig:scott triangle}
\end{figure}

We have set up the machinery to prove Scott's theorem. We  begin our discussion by invoking a fundamental edict of the Mighty Gelfand.
\begin{center}\emph{Mathematics must be done on the simplest nontrivial examples!} 
\end{center}
(See~\cite{DZ}.) For us, the key example is the right triangle that has legs of length 3, see Figure~\ref{fig:extremecase}. There is one lattice point in the interior and nine lattice points on the boundary and therefore we have the trivial inequality
$$2 \cdot\# (\textrm{interior lattice points}) - 
\#(\textrm{boundary lattice points}) +7 \geq 0 $$
for this triangle. Scott's theorem states that this relationship holds for \emph{all} convex lattice polygons that contain at least one lattice point in their interior. Furthermore, it is only when we have this particular triangle that we get equality! In all other cases we have strict inequality.

\begin{figure}
\begin{center}
\begin{tikzpicture}[scale=1]
\draw[help lines] (0,0) grid (5,5);
\draw [very thick] (1,1) --(1,4) -- (4,1) -- (1,1);
\draw[fill] (2,2) circle [radius=0.1];
\draw [fill] (1,1) circle [radius=0.1];
\draw [fill] (1,2) circle [radius=0.1];
\draw [fill] (1,3) circle [radius=0.1];
\draw [fill] (1,4) circle [radius=0.1];
\draw [fill] (2,1) circle [radius=0.1];
\draw [fill] (3,1) circle [radius=0.1];
\draw [fill] (4,1) circle [radius=0.1];
\draw [fill] (2,3) circle [radius=0.1];
\draw [fill] (3,2) circle [radius=0.1];
\end{tikzpicture}
\end{center}
\caption{The only time we get equality for Scott's theorem}
\label{fig:extremecase}
\end{figure}
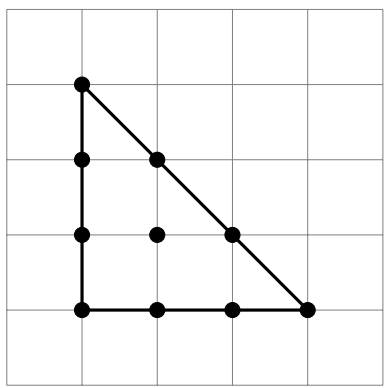

\begin{theorem}[Scott]
Let $T$ be a lattice triangle in $\R^2$ with vertices in $\Z^2$. Let 
$$I = \#(\Z^2 \cap interior(T)) \textrm{ and } B= \#(\Z^2 \cap boundary(T)),$$
with $I \geq 1$. Then,
\begin{equation}\label{eq:Scottineq}
B \leq 2I + 7.
\end{equation}
\end{theorem}

\begin{proof} Lemma~\ref{unimod-map} permits us to assume that vertices of our lattice triangle $T$ are $(0,0), (b,0)$ and $(m,h)$ with
\begin{itemize}
\item $ b,h > 0; m\ge 0; m< h;$
\item $b \ge \gcd(m,h), \gcd(m-b,h).$
\end{itemize}
We note that the base and height of $T$ are of length $b$ and $h$ respectively. Furthermore, since $T$ has an interior lattice point, $h\ge 2$. (It should be noted that the base of $T$ can be of length 1. For example it could be the triangle with vertices $(0,0), (1,0)$ and $(3,5)$.)

We now apply Pick's theorem to express the term $(2I-B)$ as 
\begin{equation}\label{eq:rwkP}
2I-B = 2(\textrm{area}(T) - B +1)=bh -2B+2.
\end{equation}
Since the number of lattice points lying on the line segment with endpoints $(x_1,y_1),(x_2,y_2) \in \Z^2$ equals 
$(\gcd(x_1-x_2,y_1-y_2) +1)$  we can rewrite~\eqref{eq:rwkP}  as
\begin{equation} 
2I-B= bh- 2(b +\gcd(m,h) + \gcd(m-b,h)) +2.
\end{equation}
On judiciously applying the inequalities 
$$\gcd(m,h) \leq b, \gcd(m,h) \leq h, \textrm{ and } \gcd(m-b,h) \leq h$$ we get the inequality 
$$2(b +\gcd(m,h) + \gcd(m-b,h)) \leq 3(b+h).$$ 
This leads us to the inequality 
\begin{equation}
2I-B \geq bh-3b-3h+2= (b-3)(h-3)-7.
\end{equation}
When $b,h \geq 3$ or $b, h \leq 2$,  we get $(b-3)(h-3) \geq 0$ and there is nothing further to prove. It remains to consider the cases when $(b-3)(h-3) < 0$. This occurs in two cases:
\begin{enumerate}
\item  $b\leq 2, h \geq 4$; 
\item $ b \geq 4, h = 2.$
\end{enumerate}

In case 1, we have that $B \leq 6$, and consequently $2I - B \geq -4.$  In case 2 we have that
$$ 2I-B = 2b -2B+2 \textrm{ and } B \leq b+4.$$
Consequently, $2I-B = 2b -2B+2 \geq -6.$ Thus in all possible cases we get the inequality $2I-B+7 \geq 0.$
\end{proof}

A small aside. We sent a draft of the paper to Paul Scott. He wrote back a lovely note where he mentioned the following: \emph{I think this lattice polygon result was the  best thing I did.  There were lots of other more serious discoveries, but I always liked this one.}

We end this section with the following (non-routine) exercise.  Scott~\cite{Sco} solved it using lattice quadrilaterals and  separating them into three different cases. The last case requires a deft touch. 

\noindent {\bf Exercise:}  Prove that Scott's inequality holds for arbitrary convex polygons. That is, prove the following: Any convex lattice polygon ${\mathcal C} \subset \R^2$, with 
$(\textrm{interior}({\mathcal C})\cap \Z^2) \not= \emptyset,$ satisfies Scott's inequality~\eqref{eq:Scottineq}. Furthermore, the only time we get equality is when ${\mathcal C}$ is an isosceles right triangle with legs of length 3.

\section{Counting Clean Triangles} 

Earlier we noted that any clean triangle has area $h/2$, with $h$ odd. By Lemma~\ref{unimod-map} such a triangle is unimodularly equivalent to a triangle with vertices $(0,0), (1,0)$ and $(m,h)$, with 
$$0\le m < h, \gcd(m,h)=\gcd(m-1,h)=1.$$ 
The question arises  whether we can find an expression for the number of inequivalent clean lattice triangles of area $h/2$.   In the late nineties the first author solved a similar question for empty tetrahedra of fixed volume~\cite{Kh1}. We apply the same method here. However there are some nuances here that were not present in the case of empty tetrahedra.

\subsection{The arithmetic function $imph$}

We start by defining a variant of the Euler phi function. Almost certainly it has arisen elsewhere and  has a well established name. Nevertheless in a spirit of whimsy we have decided to name it $imph(n)$ (imitation phi). (Unfortunately the acronym for \emph{ fake phi} is completely unsuitable!)

\begin{defn}
For $n\in N$, let 
$$ IP(n) =\{ x\in \N \; : \; 1\le x \le n, \gcd(x-1,n)=\gcd(x,n) =1 \}.$$ 
We now define the arithmetic function $imph(n)$ to be the cardinality of the set $IP(n)$, that is,
$$ imph(n) = \#IP(n).$$
\end{defn}

\begin{lemma}\label{imphmult}
$imph(n)$ is a multiplicative function.
\end{lemma}

\begin{proof}
Let $m_1,m_2\in \N$ with $\gcd(m_1,m_2)=1$. By the Chinese Remainder Theorem the map 
$f: IP(m_1m_2) \rightarrow IP(m_1) \times IP(m_2)$ via $$f(x) = (x \mod m_1, x \mod m_2)$$
 is a bijection.
\end{proof}

Our next result expresses the value of $imph(n)$ in terms of the prime factorization of $n$. The similarity to 
$\varphi(n)$ is striking.

\begin{lemma}
The function $imph(n)$ takes on the following values.
\begin{enumerate}[i.]
\item $imph(1)=1$.
\item  For any $n\in \N$, \begin{equation} imph(2^n) =0. \end{equation}
\item Let $p$ be prime. Then, for any $n\in \N$ 
\begin{equation}\label{eq:imph-p-to-n}
imph\left( p^n \right) = p^{n-1}(p-2).
\end{equation}
\item Let $p_1,p_2,\ldots, p_k$ be the distinct prime factors of $m$. Then  
\begin{equation}
imph(m) =m \prod_{i=1}^k \left(1-\frac{2}{p} \right).
\end{equation}
\end{enumerate}
\end{lemma}

The proof of~\eqref{eq:imph-p-to-n} is routine. The general result then follows by combining~\eqref{eq:imph-p-to-n} with 
Lemma~\ref{imphmult}.

\subsection{The Burnside {\it coup de grace}!}

\begin{lemma}
Let $T_m$ and $T_x$ be the triangles with vertices $(0,0),(1,0),(m,h)$ and $(0,0),(1,0),(x,h)$ respectively. If $T_x$ is unimodularly equivalent to $T_m$, then $x$ is congruent modulo $h$ to an element of the set 
$$S =  \left\{ m,m^{-1}, 1-m,1-m^{-1},(1-m)^{-1}, (1-m^{-1})^{-1} \right\}.$$
\end{lemma}

\begin{proof}
An unimodular map from $T_m$ to $T_x$ can be viewed as a bijection from the 3 element set $\{(0,0),(1,0), (m,h)\}$ to the 
3 element set $\{ (0,0), (1,0), (x,h) \}$. There are 6 bijections and each one gives an element of $S$. As an example we  show how we get $(1-m^{-1}) \in S$. On considering the bijection
$$(0,0) \mapsto (1,0), (1,0) \mapsto (x,h), (m,h) \mapsto (0,0),$$
we obtain the unimodular transformation $M\vec{v} + (1,0)$, where $M$ is the matrix 
$$ \begin{pmatrix} x-1 & (m-mx-1)/h\\ h & -m \end{pmatrix}.$$
Since $(m-mx-1)/h \in \Z$, we have the congruence $$(m-mx-1) \equiv 0 \pmod{h}.$$
It now follows that $$x \equiv 1-m^{-1} \pmod{h}.$$
\end{proof}

\subsubsection{The arithmetic function ${\mathcal T}(n)$}

\begin{defn} For a positive integer $n$, let ${\mathcal T}(n)$ be the number of unimodularly inequivalent clean triangles of area $ n/2$.
\end{defn}

Since there are no clean triangles of integral area, ${\mathcal T}(n) =0$ when $n$ is even. Our goal is to find formulae 
for ${\mathcal T}(n)$ when $n$ is odd. To do so we begin by defining 6 maps $g_i : IP(n) \rightarrow IP(n), i=1,\dots, 6.$
\begin{enumerate}
\item $ g_1 : m \mapsto m;$
\item $ g_2 : m \mapsto m^{-1} \mod n;$
\item $g_3: m \mapsto 1-m \mod n;$
\item $ g_4 : m \mapsto 1-m^{-1} \mod n;$
\item $ g_5 : m \mapsto (1-m)^{-1} \mod n;$
\item $g_6: m \mapsto (1-m^{-1})^{-1} \mod n. $
\end{enumerate}
 Let 
$$ Fix(g_i,n) = \textrm{ the set of fixed points of } g_i.$$
Then by Burnside's lemma
\begin{equation}
\mathcal{T}(n) = \frac{1}{6}\sum_{i =1}^6 \#Fix(g_i,n)
\end{equation}
To complete our task we need to determine the cardinalities of $Fix(g_i,n)$ for $i=1,\ldots,6$. The following lemma gives us some of them. 

\begin{lemma}
Let $n$ be an odd integer. We have the following:
\begin{eqnarray*}
\#Fix(g_1,n)& = & \#IP(n)= imph(n); \\
\#Fix(g_2,n)& = & \#\{x \, :\, x^2 =1 \pmod{n}, x\in IP(n)\};  \\
\#Fix(g_3,n)& = & \#\{x \, :\, x=2^{-1} \pmod{n}, x\in IP(n) \}=1; \\
\#Fix(g_4,n)& = & \#\{x \, :\, x^2-x+1 =0 \pmod{n}, x\in IP(n)\}; \\
\#Fix(g_5,n)& = & \#\{x \, :\, x^2-x+1 =0 \pmod{n}, x \in IP(n)\}; \\
\#Fix(g_6,n)& = & \#\{x \, :\, x=2 \pmod{n},x\in IP(n)\}=1.
\end{eqnarray*}
Consequently,
\begin{equation}\label{eq:initT}
{\mathcal T}(n) = \frac{imph(n) + \#Fix(g_2,n) + 2 \cdot \#Fix(g_4,n) +2}{6}.
\end{equation}
\end{lemma}

The task of determining the values of $\#Fix(g_2,n)$ and $\#Fix(g_4,n)$ necessitates a brief visit to the isles of quadratic congruences.

\subsection{A pleasant sojourn in the isles of quadratic congruences}

Two fundamental tools in elementary number theory are Hensel's lemma and the Chinese Remainder Theorem. 
By combining them we get the following two well-known results.

\begin{lemma}\label{rootsof1}
Let $p\geq 3$ be prime. Then for any $k \in \N$
\begin{equation}
\#\{x \, :\, x^2 =1 \pmod{p^k}, 1\leq x <p^k\} = 2.
\end{equation}
It follows that for any odd integer $n$,
\begin{equation}
\#\{x \, :\, x^2 =1 \pmod{n}, 1\leq x <n\} = 2^{\omega(n)},
\end{equation}
where $\omega(n) = $ the number of distinct prime divisors of n.
\end{lemma}

\begin{lemma}\label{rootsofquad}
Let $p\geq 3$ be prime. Let 
$$ N_{p^k}(x^2-x+1) = \#\{x \, :\, x^2-x+1 \equiv 0 \!\!\!\! \pmod{p^k}, 1\leq x <p^k, k \in \N\}.$$
Then, 
\begin{equation}\label{eq:noofroots}
N_{p^k}(x^2-x+1) =\left\{ \begin{array}{ll} 1,  & p=3, k=1, \\
0, & p=3, k \geq 2, \\ 2, & p \equiv 1 \pmod{6}, \\ 0, & p \equiv 5 \pmod{6}. \end{array} \right.
\end{equation}
Let $n$ be odd and let $\omega(n) = $ the number of distinct prime divisors of n.
Then from~\eqref{eq:noofroots} we can derive the following cardinalities.
\begin{enumerate}[i.]
\item If all of the prime divisors $p$ of $n$  satisfy the congruence $p\equiv 1 \pmod{6}$, then
\begin{equation}
\#\{x \, :\, x^2-x+1 \equiv 0 \pmod{n}, 1 \leq x < n\} = 2^\omega(n).
\end{equation}
\item If $n \equiv 0\pmod{3}, n\not\equiv 0 \pmod{9}$, and all of the other prime divisors $p$ of $n$ satisfy the congruence $p \equiv 1 \pmod{6}$, then
\begin{equation}
 \#\{x \, :\, x^2-x+1 \equiv 0 \pmod{n}, 1 \leq x < n\} = 2^{\omega(n)-1}
\end{equation}
\item If $n\equiv 0 \pmod{9}$ or $n$ has a prime divisor $p$ such that $p \equiv 5 \pmod{6}$, then 
\begin{equation}
 \#\{x \, :\, x^2-x+1 \equiv 0 \pmod{n}, 1 \leq x < n\} = 0.
\end{equation}
\end{enumerate}
\end{lemma}
A simple check shows that the two sets 
$$\{x \, :\, x^2-x+1 \equiv 0 \pmod{n}, 1\leq x <n\}$$
and 
$$\{x \, :\, x^2-x+1 \equiv 0 \pmod{n}, x \in IP(n)\} $$
are equal. Consequently Lemma~\ref{rootsofquad} gives us the value of $\#Fix(g_4,n)$. \emph{An incidental remark}: to prove Lemma~\ref{rootsofquad} we need to invoke the following special case for quadratic residues.
\begin{lemma} 
Let p be prime with $p>3$. Then 
\begin{equation}
\left(\frac{-3}{p}\right) = \left\{\begin{array}{cl} 1, & p\equiv 1 \pmod{6}, \\ -1, & p \equiv 5 \pmod{6}. \end{array}\right.
\end{equation}
\end{lemma}

We need to take greater care in determining the value of $\#Fix(g_2,n)$. One of the roots of $x^2-1 = 0 \pmod{p}$, with $p$ prime, equals 1. Since $1\not\in IP(p)$, we need to remove it from consideration
and rework the proof of Lemma~\ref{rootsof1}.  In doing so we get 
\begin{equation}
\#\{x \, :\, x^2 =1 \pmod{p^k}, x\in IP(p^k)\} = 1.
\end{equation}
We now invoke the Chinese Remainder Theorem to obtain 
\begin{equation}
\#Fix(g_2,n)=  \#\{x \, :\, x^2 =1 \pmod{n}, x\in IP(n)\} =1.
\end{equation}

\subsection{Putting the pieces together}

We now combine the various pieces to get a formula for ${\mathcal T}(n)$ in terms of the prime factorization of $n$.

\begin{theorem}
Let $n$ be odd and let $\omega(n) = $ the number of distinct prime divisors of n. There are three possibilities for the value of ${\mathcal T}(n)$.
\begin{enumerate}[i.]
\item If $n \equiv 0 \pmod{9}$ or $n$ has a prime divisor $p$ such that $p \equiv 5 \pmod{6}$, then
\begin{equation}
{\mathcal T}(n) = \frac{imph(n)+3}{6}.
\end{equation}
\item If $n \equiv 0 \pmod{3}$ and all of the other prime divisors $p$ of $n$ satisfy the congruence  $p \equiv 1 \pmod{6}$, then
\begin{equation}
{\mathcal T}(n) = \frac{imph(n)+2^{\omega(n)} +3}{6}.
\end{equation}
\item If all of the prime divisors $p$ of $n$ satisfy the congruence  $p \equiv 1 \pmod{6}$, then
\begin{equation}
{\mathcal T}(n) = \frac{imph(n)+2^{\omega(n)+1} +3}{6}.
\end{equation}
\end{enumerate}
\end{theorem}

Initially we ended at this juncture. However, on sending a draft to Igor Shparlinski, he suggested that calculating the mean-value of $imph(n)$ would  give us the mean-value of ${\mathcal T}(n)$. This valuable suggestion leads to a connection with the Feller-Tornier constant. 

\section{ The mean-value of ${\mathcal T}(n)$}

We begin by determining the mean-value of the arithmetical function $imph(n)/n$. This can be done using various mean-value theorems for arithmetic functions. Our chosen method is to apply the following corollary of Wintner's mean-value theorem.

\begin{theorem} 
Let $\P$ denote the set of primes in $\Z$ and let $f: \N \rightarrow \C$ be a multiplicative function. If 
\begin{equation}
\sum_{p\in \P} \frac{|f(p)-1|}{p} < \infty
\end{equation}
and 
\begin{equation}
\sum_{p\in \P} \sum_{k=2}^\infty \frac{|f(p^k|}{p^k} < \infty,
\end{equation}
then $M(f)$, the mean-value of $f$, exists and 
\begin{equation}
M(f) = \prod_{p\in \P} \left( 1+ \frac{f(p)-1}{p} + \sum_{k=2}^\infty \frac{f(p^k)-f(p^{k-1})}{p^k} \right).
\end{equation}
\end{theorem}

See~\cite[Corollary 2.3, page 51--52]{SS}. The multiplicative function $imph(n)/n$ satisfies the hypotheses of the theorem. 
Furthermore for $k \geq 2$, 
$$\frac{imph(p^k)}{p^k} - \frac{imph(p^{k-1})}{p^{k-1}} =0.$$
Consequently we obtain the following. 

\begin{cor}
The mean-value of the arithmetical function $imph(n)/n$ is 
\begin{equation}\label{eq:mean-value}
M\left(\frac{imph(n)}{n}\right) = \frac{1}{2} \prod_{p\geq 3, p\in \P}\left(1-\frac{2}{p^2}\right).
\end{equation}
\end{cor}
We note that 
\begin{equation}
\prod_{p\geq 3, p\in \P}\left(1-\frac{2}{p^2}\right) = 
 \left( 1 + \sum_{d > 1, d \textrm{ odd } }\frac{\mu(d)2^{\omega(d)}}{d^2} \right).
\end{equation}

\begin{cor}
\begin{equation}\label{eq:AV1}
\lim_{x \rightarrow \infty}\frac{\sum_{n \leq x} imph(n)}{x^2}=  \frac{1}{4} \prod_{p\geq 3, p\in \P}\left(1-\frac{2}{p^2}\right).
\end{equation}
Furthermore,
\begin{equation}\label{eq:AV2}
\lim_{x \rightarrow \infty}\frac{\sum_{n \leq x} {\mathcal T}(n)}{x^2}=  \frac{1}{24} \prod_{p\geq 3, p\in \P}\left(1-\frac{2}{p^2}\right).
\end{equation}
\end{cor}

\begin{proof}
We obtain~\eqref{eq:AV1} by combining~\eqref{eq:mean-value} with the technique of partial summation. The main contribution to the value of 
${\mathcal T}(n)$ is given by $imph(n)$ --- the contribution of $2^{\omega(n)}$ is secondary. Thus~\eqref{eq:AV2} follows from~\eqref{eq:AV1}.

To see that the contribution of $2^{\omega(n)}$ to the average order of ${\mathcal T}(n)$ is overshadowed by the contribution of $imph(n)$ we invoke a result of Grosswald. In~\cite{G} he proved that the average order of $2^{\Omega(n)}$, where $\Omega(n)$ is the total number of prime factors of $n$, 
 is $O(x\log^2(x))$.
\end{proof}

The mean-value of $imph(n)/n$ is related to an obscure number-theoretic constant called the Feller--Tornier constant~\cite{Wiki}. This constant is the infinite product 
\begin{equation}\label{eq:CFT}
C_{FT}= \frac{1}{2}+\left(\frac{1}{2} \prod_{ p\in \P}\left(1-\frac{2}{p^2}\right)\right).
\end{equation}
This number is the density of the set of all positive integers that have an even number of distinct prime factors raised to a power larger than one (ignoring any prime factors which appear only to the first power). Another product representation of $C_{FT}$ is 
$$\frac{1}{2}\left(1+\frac{1}{\zeta(2)} \prod_{ p\in \P}\left(1-\frac{1}{p^2-1}\right)\right).$$
The constant is named after its discoverers William Feller (1906--1970) and Erhard Tornier
(1894--1982).

\smallskip

\noindent {\bf Concluding Exercise.} Find a formula for the number of unimodularly inequivalent clean lattice parallelograms of fixed area.

\smallskip

\noindent {\bf Acknowledgements:} The first-named author is a great  admirer of the masterful opinion pieces of the Great and Stout-Hearted Doron Zeilberger. The discerning DEER will undoubtedly notice DZ's literary influence on this manuscript. DZ also pointed out that $imph(n)$ is not in the OEIS. Something that needs to be remedied! We thank Igor Shparlinski for his insightful suggestion to calculate the mean-value of $imph(n)$.

\end{document}